\documentclass[11pt]{article}

\usepackage{bm}
\usepackage{amsfonts}
\usepackage{amssymb,amsmath,amsthm, amsfonts}
\usepackage{mathrsfs}

\def\sqr#1#2{{\vcenter{\vbox{\hrule height.#2pt
              \hbox{\vrule width.#2pt height#1pt \kern#1pt \vrule width.#2pt}
          \hrule height.#2pt}}}}

\newcommand{\norm}[1]{\left\Vert#1\right\Vert}
\newcommand{\abs}[1]{\left\vert#1\right\vert}

\def\sqr#1#2{{\vcenter{\vbox{\hrule height.#2pt
              \hbox{\vrule width.#2pt height#1pt \kern#1pt \vrule width.#2pt}
              \hrule height.#2pt}}}}
\def\3n{\negthinspace \negthinspace \negthinspace }
\def\2n{\negthinspace \negthinspace }
\def\1n{\negthinspace }

\def\={\buildrel \triangle \over =}

\def\max{\mathop{\rm max}}
\def\min{\mathop{\rm min}}

\def\({\Big (}
\def\){\Big )}
\def\[{\Big[}
\def\]{\Big]}
\def\be{\begin{equation}}

\def\ee{\end{equation}}

\def\square#1{\vbox{\hrule\hbox{\vrule height#1%
     \kern#1\vrule}\hrule}}
\def\rectangle#1#2{\vbox{\hrule\hbox{\vrule height#1%
     \kern#2\vrule}\hrule}}

\newenvironment{eqs}{\equation\aligned}{\endaligned\endequation}
\newenvironment{eqc}{\equation\begin{cases}}{\end{cases}\endequation}
\providecommand{\norm}[1]{{\left\lVert#1\right\rVert}}
\providecommand{\abs}[1]{{\left\lvert#1\right\rvert}}
\providecommand{\pr}[1]{{\left(#1\right)}}
\providecommand{\pp}[1]{{\left[#1\right]}}
\providecommand{\set}[1]{{\left\lbrace#1\right\rbrace}}

\font\tenbb=msbm10 \font\sevenbb=msbm7 \font\fivebb=msbm5

\newfam\bbfam
\scriptscriptfont\bbfam=\fivebb \textfont\bbfam=\tenbb
\scriptfont\bbfam=\sevenbb

\theoremstyle{definition}
\newtheorem{lemma}{Lemma}[section]
\newtheorem{remark}{Remark}[section]
\newtheorem{example}{Example}[section]
\newtheorem{theorem}{Theorem}[section]

\newtheorem{definition}{Definition}[section]
\newtheorem{prop}{Proposition}[section]

\newtheorem{ass}{Assumption}[section]

\makeatletter
   
   \@addtoreset{equation}{section}
\makeatother

\begin{document}

\title{Controllability concepts for mean-field dynamics with reduced-rank coefficients\thanks{This work was supported by the NSF of Shandong Province (NO. ZR2023ZD35), the NSF of P.R. China (NO. 12031009), National Key R and D Program of China (NO. 2018YFA0703900).}}
\author{Dan Goreac$^{1,3,4}$,\,\, Juan Li$^{1,2,}$\footnote{Corresponding authors.},\,\, Xinru Zhang$^{1,\dag}$ \\
{$^1$\small School of Mathematics and Statistics, Shandong University, Weihai, Weihai 264209, P.~R.~China.}\\
{$^2$\small Research Center for Mathematics and Interdisciplinary Sciences, Shandong University,}\\
{\small Qingdao 266237, P.~R.~China.}\\
{$^3$\small \'{E}cole d'actuariat, Universit\'{e} Laval, Qu\'{e}bec, QC, Canada.}\\
{$^4$\small LAMA, Univ Gustave Eiffel, UPEM, Univ Paris Est Creteil, CNRS, F-77447, Marne-la-Vall\'ee, France.}\\
{\footnotesize{\it E-mails: dan.goreac@univ-eiffel.fr,\,\ juanli@sdu.edu.cn,\,\ xinruzhang@mail.sdu.edu.cn.}}
}
\date{\today}
\maketitle

\textbf{Abstract}. In this paper we explore several novel notions of exact controllability for mean-field linear controlled stochastic differential equations (SDEs). A key feature of our study is that the noise coefficient is not required to be of full rank. We begin by demonstrating that classical exact controllability with $\mathbb{L}^2$-controls necessarily requires both rank conditions on the noise introduced in \cite{goreac2014controllability} and subsequent works. When these rank conditions are not satisfied, we introduce alternative rank requirements on the drift, which enable exact controllability by relaxing the regularity of the controls. In cases where both the aforementioned rank conditions fail, we propose and characterize a new notion of exact terminal controllability to normal laws (ETCNL). Additionally, we investigate a new class of Wasserstein-set-valued backward SDEs that arise naturally associated to ETCNL.


\section{Introduction}
\indent Let $(\Omega,\mathcal{F},\mathbb{P})$ be a probability space endowed with a scalar standard Brownian motion $W(\cdot)$ and its natural filtration $\mathbb{F}=\{\mathcal{F}_{t}\}_{t\geq0}$ augmented by all $\mathbb{P}$-null sets. We consider the following mean-field linear controlled stochastic differential equation (MFSDE, for short):
\begin{eqs}\label{SDE0}
&dX(t)=\pr{A_1^0X(t)+A_2^0\mathbb{E}\pp{X(t)}+B_1u(t)+B_2\mathbb{E}\pp{u(t)}}dt\\&\indent \quad \quad +\pr{C_1^0X(t)+C_2^0\mathbb{E}\pp{X(t)}+D_1u(t)+D_2\mathbb{E}\pp{u(t)}}dW(t),
\end{eqs}
where $A_j^0,\ C_j^0\in\mathbb{R}^{d\times d}$ and $B_j,\ D_j\in\mathbb{R}^{d\times n}$ for every index $j\in\set{1,2}$. Here, $\mathbb{R}^{d\times d}$ and $\mathbb{R}^{d\times n}$ are the sets of all $(d\times d)$ and $(d\times n)$ real matrices, respectively.
In the above, the state space is a standard Euclidean one $\mathbb{R}^d$, with $d\in\mathbb{N}^*$, while the controls are $n$-dimensional, for some $n\in\mathbb{N}^*$. The concept of exact controllability describes the ability of a system to reach any given terminal point from an arbitrary initial point accurately, and the precise definition will be given in the next section.\\
\indent Referring strictly to this kind of mean-field systems, exact controllability with $\mathbb{L}^2$-controls is shown to require \begin{equation}\label{nec0}
rank(D_1+D_2)=d,
\end{equation} and a sufficient condition is obtained by adjoining
\begin{equation}\label{suf0}rank(D_1)=d,
\end{equation} cf. \cite{goreac2014controllability}, condition employed in subsequent papers, e.g., \cite[Assumption H]{yu2021controllability}, \cite[Assumption H4]{chen2023exact}.
The standard reasoning is that, given \eqref{suf0}, there exists $M\in \mathbb{R}^{n\times d}$ such that $D_1M=C_1^0$ and the equation \eqref{SDE0} and its qualitative features are equivalent to those of
\begin{eqs}\label{SDE1}
&dX(t)=\pr{A_1X(t)+A_2\mathbb{E}\pp{X(t)}+B_1u(t)+B_2\mathbb{E}\pp{u(t)}}dt\\&\indent \quad \ \ \ \ +\pr{C\mathbb{E}\pp{X(t)}+D_1u(t)+D_2\mathbb{E}\pp{u(t)}}dW(t),\ t\in[0,T],
\end{eqs}with $A_k=A_k^0-B_kM, \ k\in\set{1,2},\ C=C_2^0-D_2M$. Therefore, in this paper we focus on the controllability of \eqref{SDE1} and give an affirmative answer to the following questions:
\begin{enumerate}
\item[\textbf{Q1.}] Is \eqref{suf0} a necessary condition for the exact controllability of \eqref{SDE1} with $\mathbb{L}^2$-controls?
\item[\textbf{Q2.}] When both $D_1=0=D_2$, can the exact controllability be achieved by lowering the regularity of the controls and imposing rank conditions on $B_1,B_2$?
\item[\textbf{Q3.}] When $D_1=0$ (hence \eqref{suf0} fails to hold), but $B_1$ does not comply with the previous framework, does the condition \eqref{nec0}, i.e., $rank(D_2)=d$ still correspond to a mean-field notion of exact (terminal) controllability? If so, is there a novel notion of backward SDE/inclusion that can be associated to this?
\end{enumerate}
\subsection{State of the art}
\indent Since it deals with targeted behavior of systems governed by ordinary or partial differential dynamics, the controllability is a concept that has been intensively investigated. In the framework of deterministic control systems, the controllability of linear systems is characterized by the controllability Gramian and Kalman's rank condition, respectively. For comprehensive surveys on controllability results for non-random systems, see, e.g., Lee and Markus \cite{lee1967foundations} for ordinary differential equation (ODE, for short) systems and Russell \cite{russell1978controllability}, Lions \cite{lions1988exact} and Zuazua \cite{zuazua2007controllability} for partial differential equation (PDE, for short) systems.\\
\indent Within the extensive literature on stochastic differential equation (SDE), we would like to highlight the following results that are particularly relevant to our current research. Peng \cite{Peng1994backward} initially introduced the concepts of exact terminal controllability and exact controllability in $\mathbb{L}^2$ for stochastic control systems. In the case where the coefficients are deterministic and time-invariant, Peng derived a necessary condition for the terminal controllability of stochastic control systems. Additionally, for linear stochastic control systems, a purely algebraic necessary and sufficient criterion was introduced. Later on, Liu and Peng \cite{liu2010controllability} expanded his research findings to include linear SDE systems with bounded, time-varying deterministic coefficients, leading to the development of a random version of the controllability Gramian. L{\"{u}} and Zhang \cite{lu2021mathematical} simplified Peng's rank condition from infinite number of columns to finite number of columns and derived a stochastic Kalman-type rank condition. In a related study, presented in \cite{lu2012representation}, L{\"{u}}, Yong and Zhang put forth the representation of the It\^{o}'s integral as Lebesgue/Bochner integral, which proved to be a valuable tool. Wang et al. \cite{wang2016exact} introduced the notion of $\mathbb{L}^{p}$-exact controllability for linear controlled stochastic differential equations with random coefficients. The authors then established several sufficient conditions for achieving this exact controllability, as well as deriving several equivalent conditions. On the other hand, Buckdahn, Quincampoix and Tessitore \cite{buckdahn2006characterization} and Goreac \cite{goreac2008kalman, goreac2021improved} delved into the concept of approximate controllability, specifically investigating the existence of admissible controls that steer trajectories arbitrarily close to the terminal conditions. In particular, in \cite{goreac2008kalman} the author showed that approximate and exact controllability are not equivalent.\\
\indent Mean-field stochastic differential equation (MFSDE, for short), also known as the McKean-Vlasov equation, can be traced back to the work of Kac \cite{kac1956foundations, kac1959probability} in the 1950s. With the pioneering works on mean-field stochastic differential games by
Lasry and Lions \cite{lasry2007mean}, the study of mean-field problems has garnered significant attention from researchers, see, e.g., \cite{buckdahn2009mean}, \cite{buckdahn2011general}, \cite{buckdahn2009mean1} as well as the references therein. At the same time, the results on the controllability of McKean-Vlasov SDEs are rather scarce. Goreac \cite{goreac2014controllability} studied some exact and approximate controllability properties for mean-field linear stochastic systems with deterministic time-invariant coefficients, showing that \eqref{nec0} is a necessary condition and, together with \eqref{suf0}, this is a sufficient set of conditions guaranteeing exact controllability. Yu \cite{yu2021controllability} obtained a mean-field version of the Gramian matrix criterion for the general time-variant case and a mean-field version of the Kalman rank condition for the special time-invariant case, also in the deterministic framework. Ye and Yu \cite{ye2022exact} were interested in the exact controllability of linear stochastic systems (but with random time-variant coefficients) and obtained the equivalence for a series of problems. Recently, Chen and Yu \cite{chen2023exact} considered the exact controllability of a linear mean-field type game-based control system generated by a linear-quadratic Nash game. In an effort to deal with the exact controllability, the aforementioned papers require the conditions \eqref{nec0} and \eqref{suf0} or some variations.
\subsection{Structure of the paper and main contributions}
\indent The paper is organized as follows. In Section \ref{Sec2}, we provide all the necessary preparations, including some spaces and controllability concepts. Section \ref{Sec3} is devoted to the condition on the noise coefficients. We show the necessity of the full rank of the $D_1$ operator, thus answering \textbf{Q1}.
Section \ref{Sec4} tackles \textbf{Q2} in Theorem \ref{ThMainQ2}. We show that under a suitable condition on the drift coefficients $B_1$, and $B_1+B_2$, even when \eqref{suf0} fails to hold, the system is still exactly controllable albeit the weakening of controls. When both $B_1=D_1=0$, we show that $Rank(D_2)=d$ is a necessary and sufficient condition for a new concept of exact terminal controllability to normal laws (ETCNL), thus answering \textbf{Q3} in Theorem \ref{ThMainQ3}. The novelty of ETCNL with respect to existing controllability notions is explained in Section \ref{Sec5c}.
In the remaining of Section \ref{Sec5}, we focus on the notion of BSDEs with terminal Gaussian law that is naturally associated to ETCNL. In particular, we produce Example \ref{exp2} showing that the solutions need to be Wasserstein set-valued and give a complete characterization in the one-dimensional case in Theorem \ref{BSDE_1}.
\section{Preliminaries}\label{Sec2}
\indent We let $T>0$ be a fixed finite time horizon and $\mathbb{R}^{n\times m}$ denote the set of all $(n\times m)$ matrices. The superscript $^*$ denotes the transpose of a vector or a matrix. We introduce the following functional spaces, when the state space $H$ is a Hilbert one.\\
\noindent 1) For $p\in\pp{1,\infty}$, $\mathbb{L}^p_{\mathcal{F}_T}\pr{\Omega;H}$ denotes the space of $\mathcal{F}_T$-measurable, $H$-valued random variables $\xi$ such that $\mathbb{E}\pp{\norm{\xi}^p}<\infty$, for $p<\infty$, and ${\rm esssup}_{\omega\in\Omega}\norm{\xi(\omega)}<\infty$, for $p=+\infty$.\\
\noindent 2) For $p,\ q\in[1,\infty]$,\\
(a) $\mathbb{A}^{p,q}(H):=\mathbb{L}^p_{\mathbb{F}}\pr{\Omega;\mathbb{L}^q\pr{\pp{0,T};H}}$ denotes the family of progressively measurable processes $\phi$ such that $$\norm{\omega\mapsto \norm{\phi(\cdot,\omega)}_{\mathbb{L}^q\pr{\pp{0,T};H}}}_{\mathbb{L}^p(\Omega;\mathbb{R})}<\infty. $$
(b) $\mathbb{B}^{q,p}(H):=\mathbb{L}^q_{\mathbb{F}}\pr{\pp{0,T};\mathbb{L}^p\pr{\Omega;H}}$ denotes the family of progressively measurable processes $\phi$ such that $$\norm{t\mapsto \norm{\phi(t,\cdot)}_{\mathbb{L}^p(\Omega;H)}}_{\mathbb{L}^q(\pp{0,T};\mathbb{R})}<\infty. $$
\noindent 3) For $p\in\pp{1,\infty}$, $u(\cdot)\in\mathcal{U}^{p}$ (resp., $v\in\mathcal{U}^{p}_{r}\pr{\pp{0,T}}$) is a progressively measurable control such that $B_{1}u(\cdot)\in\mathbb{A}^{p,1}(H)$, and $ B_{2}\mathbb{E}\pp{u(\cdot)}\in\mathbb{L}^1\pr{\pp{0,T};H}$ (resp., $B_{1}v(\cdot)\in\mathbb{B}^{1,p}(H)$, and $ B_{2}\mathbb{E}\pp{v(\cdot)}\in\mathbb{L}^1\pr{\pp{0,T};H}$).\\
\noindent 4) For $p=q=2$, we simply write $\mathbb{L}^2_{\mathbb{F}}(H):=\mathbb{L}^2_{\mathbb{F}}\pr{\Omega\times\pp{0,T};H}$.\\
\noindent 5) Finally, for a random variable $\xi\in \mathbb{L}^2_{\mathcal{F}_T}(\Omega;H)$, we let $\mathbb{P}_{\xi}$ denote its law.

The classical controllability concepts for a $d$-dimensional space SDE (see, for instance, \cite{Peng1994backward}) are concerned with $\mathbb{L}^2_{\mathcal{F}_T}\pr{\Omega;\mathbb{R}^d}$-targets and employ $\mathbb{L}^2_{\mathbb{F}}(\mathbb{R}^n)$-regular control policies. The extensions to $\mathbb{L}^p_{\mathcal{F}_T}\pr{\Omega;\mathbb{R}^d}$-targets and different classes of controls are straightforward.
\begin{definition}
1) The system \eqref{SDE1} is said to be $\mathbb{L}^2$-\emph{\textbf{exactly terminal controllable}} if, for every target $\xi\in \mathbb{L}^2_{\mathcal{F}_T}\pr{\Omega;\mathbb{R}^d}$, one is able to find $x\in \mathbb{R}^d$ and an $\mathbb{L}^2_{\mathbb{F}}\pr{\mathbb{R}^n}$-regular control $u$ such that the solution $X^{x,u}$ of \eqref{SDE1} starting at $x$ and controlled with $u$ satisfies $X^{x,u}(T)=\xi,\ \mathbb{P}$-a.s.;\\
\indent 2) The system \eqref{SDE1} is said to be $\mathbb{L}^2$-\emph{\textbf{exactly controllable}} if, for every target $\xi\in \mathbb{L}^2_{\mathcal{F}_T}\pr{\Omega;\mathbb{R}^d}$ and every $x\in \mathbb{R}^d$, there is an $\mathbb{L}^2_{\mathbb{F}}\pr{\mathbb{R}^n}$-regular control $u$ such that the solution $X^{x,u}$ of \eqref{SDE1} starting at $x$ and controlled by $u$ satisfies $X^{x,u}(T)=\xi,\ \mathbb{P}$-a.s.;\\
\indent 3) The system \eqref{SDE1} is said to be $\mathbb{L}^2$-\emph{\textbf{exactly terminal null-controllable}} (resp. \emph{\textbf{exactly null-controllable}}) if the aforementioned properties are valid for the target $\xi=0$;\\
\indent 4) The system \eqref{SDE1} is said to be $\mathbb{L}^2$-\emph{\textbf{approximately controllable}} if, for every $\varepsilon>0$, for every target $\xi\in \mathbb{L}^2_{\mathcal{F}_T}\pr{\Omega;\mathbb{R}^d}$, and every $x\in \mathbb{R}^d$, there exists an $\mathbb{L}^2_{\mathbb{F}}\pr{\mathbb{R}^n}$-regular control $u$, such that $\mathbb{E}\pp{\abs{X^{x,u}(T)-\xi}^2}\leq \varepsilon$.
\end{definition}
\begin{remark}
1) The notions of $\mathbb{L}^2$-controllability involve $\mathbb{L}^2_{\mathbb{F}}\pr{\mathbb{R}^n}$-regular controls and this regularity of the controls is not specified in the statements, unless we particularly want to emphasize that the proof specifically relies on this regularity;\\
\indent 2) In the same spirit, one can consider notions of $\mathbb{L}^p$-controllability with controls $u\in \mathcal{A}$ (a generic class of controls). For instance, if for every initial state $x\in \mathbb{R}^d$, and every terminal target $\xi\in \mathbb{L}^p_{\mathcal{F}_T}\pr{\Omega;\mathbb{R}^d}$, there exists a control $u\in \mathcal{U}^{p}_{r}$ such that the solution $X^{x,u}$ of \eqref{SDE1} satisfies $X^{x,u}(T)=\xi,\ \mathbb{P}$-a.s., then the system is said to be exactly controllable to $\mathbb{L}^p_{\mathcal{F}_T}\pr{\Omega;\mathbb{R}^d}$-targets using $ \mathcal{U}^{p}_{r}$-controls.
\end{remark}
\section{\textbf{Q1.} Necessity of \eqref{suf0}}\label{Sec3}
This section is presented as a complement to \cite{goreac2014controllability} and provides an affirmative answer to \textbf{Q1}. Building on the work in  \cite[inequality (8)]{Peng1994backward}, which relies on a deterministic oscillating function, it has been established in \cite{goreac2014controllability} that $\mathbb{L}^2$-exact terminal controllability cannot be achieved unless $D_1+D_2$ covers the entire state space. When this necessary condition is complemented with the requirement that $D_1$ can span the state space, the two conditions together imply the exact terminal controllability. In other words, when $D_1$ is of full rank, the forward controlled equation \eqref{SDE1} can be interpreted as a classical BSDE, which is solvable. This condition has subsequently been utilized in several follow-up papers, e.g., \cite[Assumption H]{yu2021controllability}, \cite[Assumption H4]{chen2023exact}.
\begin{theorem}\label{PropNec}
Assume that $Range (C)\subset Range (D_1)$. Then \eqref{SDE1} is $\mathbb{L}^2$-exactly terminal controllable if and only if $Rank(D_1+D_2)=Rank(D_1)=d$.
\end{theorem}
\begin{proof}
We begin our proof with a simple remark. Let $\alpha$ be a real-valued process on $\Omega\times\pp{0,T}$. Then, for every Borel function $r:\pp{0,T}\longrightarrow\mathbb{R}$, we have
$$\mathbb{E}\pp{\int_t^T\abs{\alpha_s-r_s}^2ds}\geq \int_t^TVar\pr{\alpha_s}ds.$$In particular, for $0\leq t\leq T$,
\begin{equation}\label{eq_L1}\mathbb{E}\pp{\int_t^T\abs{\frac{\sqrt{2}}{\sqrt{T}}W(s)-r_s}^2ds}\geq \frac{T^2-t^2}{T}\geq T-t.
\end{equation}
We only prove the necessity of the condition on $D_1$, the sufficiency and the rank condition on $D_1+D_2$ already appearing in \cite[the Propositions 2 and 5]{goreac2014controllability}. We proceed by contradiction and assume the range of $D_1$ does not cover all $\mathbb{R}^d$. Then, there exists some $a\in\mathbb{R}^d$, such that $\abs{a}=1$ and $a^*D_1=0$. Under the assumption of exact terminal controllability, there exists some control $u$ such that the associated trajectory satisfies
$$X(T)=\frac{W(T)^2-T}{\sqrt{2T}}a=\pr{\int_0^T\sqrt{\frac{2}{T}}W(s)dW(s)}a.$$
We have
\begin{align*}
&\int_0^T\sqrt{\frac{2}{T}}W(s)dW(s)=a^*x+\int_0^Ta^*\pr{A_1X(s)+A_2\mathbb{E}\pp{X(s)}+B_1u(s)+B_2\mathbb{E}\pp{u(s)}}ds\\
&\indent \quad \hspace{2.8cm}+\int_0^Ta^*\pr{C\mathbb{E}\pp{X(s)}+D_2\mathbb{E}\pp{u(s)}}dW(s).
\end{align*}
By taking the conditional expectation $\mathbb{E}\pp{\ \cdot \mid \mathcal{F}_t}$, we get
\begin{align*}
&\int_t^T\pp{\sqrt{\frac{2}{T}}W(s)- a^*\pr{C\mathbb{E}\pp{X(s)}+D_2\mathbb{E}\pp{u(s)}}}dW(s)\\
&=\int_t^Tb(s)ds-\mathbb{E}\pp{\int_t^Tb(s)ds \Big| \mathcal{F}_t},\ t\in[0,T],
\end{align*}
where
$b(s):=a^*\Big(A_1X(s)+A_2\mathbb{E}\pp{X(s)}+B_1u(s)+B_2\mathbb{E}\pp{u(s)}\Big),\ s\in[0,T]$.
By using the first remark in our proof, more precisely, \eqref{eq_L1}, it follows that
$$\mathbb{E}\pp{\pr{\int_t^Tb(s)ds-\mathbb{E}\pp{\int_t^Tb(s)ds \Big| \mathcal{F}_t}}^2}\geq T-t,$$for every $0\leq t\leq T$. By Cauchy's inequality, the left-hand term is dominated by $(T-t)\mathbb{E}\pp{\displaystyle\int_t^T\abs{b(s)}^2ds}$. Therefore, we have $\mathbb{E}\pp{\displaystyle\int_t^T\abs{b(s)}^2ds}\geq 1$, for every $0\leq t< T$. In virtue of the square integrability of $b$, we get immediately a contradiction by taking the limit as $t\uparrow T$. We conclude that $D_1$ must have a full $d$ rank.
\end{proof}

The result in Theorem \ref{PropNec} is, of course, very strong as it states that one still needs access to all the components (via $D_1$) and that $D_2$ does not bring anything new (it is obvious that the system is exactly terminal controllable with $D_2=0$). On the contrary, $D_2$ has to ``play nice" and not affect the spectrum of $D_1$ (this is the meaning of $D_1+D_2$ being of full rank).
\section{\textbf{Q2.} Exact controllability with controls of lower regularity}\label{Sec4}
We recall the following interesting result from \cite[Theorem 3.1]{lu2012representation}.
\begin{lemma}\label{representation}
Let $H$ be a Hilbert space, $p\geq1$. Then, for every $T>0$ and every $\xi\in\mathbb{L}^p_{\mathcal{F}_T}\pr{\Omega;H}$, there exists $\zeta\in\mathbb{L}^1_{\mathbb{F}}\pr{\pp{0,T};\mathbb{L}^p\pr{\Omega;H}}$ such that $\displaystyle\int_0^T\zeta(r)\ dr=\xi$, $\mathbb{P}$-a.s.
\end{lemma}

In this section we deal with the following controlled MFSDE, for which we have assumed that $D_1=D_2=0$, thus failing to be covered by the results in the previously cited references.
\begin{equation}\label{x}
\left\{\begin{aligned}
&d X(t)= \left(A_1 X(t)+A_2 \mathbb{E}\pp{X(t)}+B_1 u(t)+B_2 \mathbb{E}\pp{u(t)}\right)dt+C \mathbb{E}\pp{X(t)} d W(t),\ t \geq 0 .\\
&X(0)=x \in \mathbb{R}^d.
\end{aligned}\right.
\end{equation}
With the notation $\bar{Y}:=\begin{pmatrix}
Y \\
y
\end{pmatrix}=\begin{pmatrix}
X-\mathbb{E}\pp{X} \\
\mathbb{E}\pp{X}
\end{pmatrix}$, and $\bar{u}:=\begin{pmatrix}
u-\mathbb{E}\pp{u} \\
\mathbb{E}\pp{u}
\end{pmatrix}$, we have the following equivalent equation
\begin{equation}\label{Y}
\left\{\begin{aligned}
&d\bar{Y}(t)= \pr{\bar{A}\bar{Y}(t)+\bar{B} \bar{u}(t)}dt+\bar{C} \bar{Y}(t)d W(t),\ t \geq 0, \\
&\bar{Y}(0)= \left(\begin{array}{l}
0\\
x
\end{array}\right).\end{aligned}\right.
\end{equation}
To shorten notations, we denote the matrix coefficients in \eqref{Y} by
$$\bar{A}:=\left(\begin{array}{cc}
A_1 & 0 \\
0 & A_1+A_2
\end{array}\right),\ \bar{B}:=\left(\begin{array}{cc}
B_1 & 0 \\
0 & B_1+B_2
\end{array}\right) \text{and}\ \bar{C}:=\left(\begin{array}{ll}
0 & C \\
0 & 0
\end{array}\right).$$

We begin with a simple result on the integrability of $\phi-\mathbb{E}\pp{\phi}$ in the cases presented at the beginning.
\begin{prop}\label{Prop1}
1) For $p \in [1,\infty],\ and\ \xi \in \mathbb{L}^p_{\mathcal{F}_T}\pr{\Omega;H}$, one has $\xi-\mathbb{E}\pp{\xi} \in \mathbb{L}^p_{\mathcal{F}_T}\pr{\Omega;H}$;\\
2) For $p,\ q \in [1,\infty]$,\\
\indent \quad (a) Whenever $\phi\in\mathbb{L}^p_{\mathbb{F}}\pr{\Omega;\mathbb{L}^q\pr{\pp{0,T};H}}$, it follows that
$$\phi-\mathbb{E}\pp{\phi}\in\mathbb{L}^p_{\mathbb{F}}\pr{\Omega;\mathbb{L}^{min\{p,q\}}\pr{\pp{0,T};H}}. $$
\indent \quad (b) Whenever $\phi\in\mathbb{L}^q_{\mathbb{F}}\pr{\pp{0,T};\mathbb{L}^p\pr{\Omega;H}}$, it follows that
$$\phi-\mathbb{E}\pp{\phi}\in\mathbb{L}^q_{\mathbb{F}}\pr{\pp{0,T};\mathbb{L}^p\pr{\Omega;H}}.$$
\indent \quad (c) Whenever $v\in\mathcal{U}^{p}\pr{\pp{0,T}}$ (resp. $v\in\mathcal{U}^{p}_{r}\pr{\pp{0,T}}$), then
$$v-\mathbb{E}\pp{v}\in\mathcal{U}^{p}\pr{\pp{0,T}} (\text{resp. }v-\mathbb{E}\pp{v}\in\mathcal{U}^{p}_{r}\pr{\pp{0,T}}).$$
\end{prop}
\begin{proof}1) One simply writes $\|\mathbb{E}\pp{\xi} \|^{p}_{H}\leq\pr{\mathbb{E}\pp{\|\xi\|_{H}}}^{p}\leq\mathbb{E}\pp{\|\xi\|_{H}^{p}} $ to conclude.\\
2)-(a) Let $p<\infty,q<\infty$ and $\phi\in\mathbb{L}^p_{\mathbb{F}}\pr{\Omega;\mathbb{L}^q\pr{\pp{0,T};H}}$.
If $p\leq q$, then, using H{\"{o}}lder's inequality,
\begin{align*}
\int_0^T\|\mathbb{E}\pp{\phi(t)}\|_{H}^{p}dt\leq &\mathbb{E}\pp{\int_0^T\|\phi(t)\|_{H}^{p}dt} \leq\mathbb{E}\pp{\pr{\int_0^T\|\phi(t)\|_{H}^{q} dt}^\frac{p}{q}}T^{1-\frac{p}{q}}<\infty.
\end{align*}
If $p> q$, then, by the monotonicity of norms,
\begin{align*}
\int_0^T\|\mathbb{E}\pp{\phi(t)}\|_{H}^{q}dt\leq &\mathbb{E}\pp{\int_0^T\|\phi(t)\|_{H}^{q}dt} \leq\pr{\mathbb{E}\pp{\pr{\int_0^T\|\phi(t)\|_{H}^{q}dt}^\frac{p}{q}}}^{\frac{q}{p}}<\infty.
\end{align*}
The conclusion is extended to the cases when $\{p,q\}\ni+\infty$ using the expression of $\mathbb{L}^\infty$ norms as non-decreasing limits of $\mathbb{L}^m$ ones, as $m\rightarrow\infty$.\\
2)-(b) Again, we can reduce our analysis to the cases when $p<\infty,q<\infty$. We let $\phi\in\mathbb{L}^q_{\mathbb{F}}\pr{\pp{0,T};\mathbb{L}^p\pr{\Omega;H}}$. One uses $\|\mathbb{E}\pp{\phi(t)}\|_{H}^{p}\leq\mathbb{E}\pp{\|\phi(t)\|_{H}^{p}}$ to conclude that
$$\int_0^T\pr{\|\mathbb{E}\pp{\phi(t)}\|_{H}^{p}}^{\frac{q}{p}}dt\leq\int_0^T\pr{\mathbb{E}\pp{\|\phi(t)\|_{H}^{p}}}^{\frac{q}{p}}dt<\infty.$$
2)-(c) The reader is reminded that $u(\cdot)\in\mathcal{U}^{p}$ is a progressively measurable control such that
$$B_{1}u(\cdot)\in\mathbb{L}^p_{\mathbb{F}}\pr{\Omega;\mathbb{L}^1\pr{\pp{0,T};\mathbb{R}^{d}}}\ \text{and}\ B_{2}\mathbb{E}\pp{u(\cdot)}\in\mathbb{L}^1\pr{\pp{0,T};\mathbb{R}^{d}}.$$
By assertion 2)-(a),
$$B_{1}(u(\cdot)-\mathbb{E}\pp{u(\cdot)})=B_{1}u(\cdot)-\mathbb{E}\pp{B_{1}u(\cdot)}\in\mathbb{L}^p_{\mathbb{F}}\pr{\Omega;\mathbb{L}^1\pr{\pp{0,T};\mathbb{R}^{d}}},$$
which is equivalent to $u(\cdot)-\mathbb{E}\pp{u(\cdot)}\in\mathcal{U}^{p}$.

The assertion on $\mathcal{U}_{r}^{p}$ is quite similar and based on 2)-(b).
\end{proof}
\begin{remark}
If $\bar{B}$ is of full rank ($2d$), then
$\mathcal{U}^{p}=\mathbb{L}^p_{\mathbb{F}}\pr{\Omega;\mathbb{L}^1\pr{\pp{0,T};\mathbb{R}^{2n}}}$, and similarly, $\mathcal{U}^{p}_{r}=\mathbb{L}^1_{\mathbb{F}}\pr{\pp{0,T};\mathbb{L}^p\pr{\Omega;\mathbb{R}^{2n}}}$.
\end{remark}

The following result shows that, by weakening the regularity of the control processes, we can guarantee the controllability of $\left\{\left(\xi-\mathbb{E}\pp{\xi},\mathbb{E}\pp{\xi}\right),\xi \in \mathbb{L}^p_{\mathcal{F}_T}\pr{\Omega;H}\right\}$.
\begin{lemma}\label{Lem1}
If $Rank(B_{1}+B_{2})=Rank(B_{1})=d$, and $p\geq1$, and $\xi \in \mathbb{L}^p_{\mathcal{F}_T}\pr{\Omega;H}$, then, there exists $u\in\mathbb{L}^1_{\mathbb{F}}\pr{\pp{0,T};\mathbb{L}^p\pr{\Omega;H}}$ such that
$$\int_0^T\bar{B}\left(\begin{array}{c}
u(t)-\mathbb{E}\pp{u(t)} \\
\mathbb{E}\pp{u(t)}
\end{array}\right)dt=\left(\begin{array}{c}
\xi-\mathbb{E}\pp{\xi} \\
\mathbb{E}\pp{\xi}
\end{array}\right).$$
\end{lemma}
\begin{proof}
We recall that $\bar{B}=\left(\begin{array}{cc}
B_1 & 0 \\
0 & B_1+B_2
\end{array}\right).$ Due to Lemma \ref{representation}, there exists $v\in\mathbb{L}^1_{\mathbb{F}}\pr{\pp{0,T};\mathbb{L}^p\pr{\Omega;H}}$, such that
$$\displaystyle\int_0^Tv(t)dt=\xi,\ \mathbb{P}\text{-a.s.}$$
We now define
\begin{align*}u(\cdot)&:=B_{1}^*(B_{1}B_{1}^*)^{-1}(v(\cdot)-\mathbb{E}\pp{v(\cdot)})+\frac{1}{T}(B_{1}+B_{2})^*[(B_{1}+B_{2})(B_{1}+B_{2})^*]^{-1}\mathbb{E}\pp{\xi},\end{align*}
and note that $u(\cdot)\in\mathbb{L}^1_{\mathbb{F}}\pr{\pp{0,T};\mathbb{L}^p\pr{\Omega;H}}$ due to Proposition \ref{Prop1}, assertion 2)-(b) (take also a look at assertion (c) dealing with matrix multiplied controls).\\
We have\\
$$\mathbb{E}\pp{u(t)}=\frac{1}{T}(B_{1}+B_{2})^*[(B_{1}+B_{2})(B_{1}+B_{2})^*]^{-1}\mathbb{E}\pp{\xi}$$
so that
$$\displaystyle\int_0^T(B_{1}+B_{2})\mathbb{E}\pp{u(t)}dt=\mathbb{E}\pp{\xi}.$$
Furthermore, $B_{1}\pr{u(t)-\mathbb{E}\pp{u(t)}}=v(t)-\mathbb{E}\pp{v(t)}$ so that
$$\int_0^TB_{1}\pr{u(t)-\mathbb{E}\pp{u(t)}}dt=\xi-\mathbb{E}\pp{\xi}.$$
The claim is now proven.
\end{proof}

\indent We consider the fundamental matrix equation
\begin{equation}\label{ODE_1}
\left\{\begin{aligned}
&d\mathcal{Y}(t)=-\mathcal{Y}(t)\bar{A}dt,\ t\in[0,T],\\
&\mathcal{Y}(0)=I_{2d},
\end{aligned}\right.
\end{equation}
whose solution is obviously deterministic and belongs to $\mathbb{L}^\infty\pr{\pp{0,T};\mathbb{R}^{d\times d}}$. In particular, note that
$$\mathcal{Y}(t)=\left(\begin{array}{cc}
\mathcal{Y}_1(t):=e^{-t A_1} & 0 \\
0 & \mathcal{Y}_2(t):=e^{-t (A_1+A_2)}
\end{array}\right),\ t\in[0,T],$$ is block-diagonal.
With the notation $$\bar{X}(t):=\left(\begin{array}{c}
X(t) \\
x(t)=\mathbb{E}\pp{X(t)}
\end{array}\right),\ t\in[0,T],$$ the autonomous deterministic system
$$dx(t)=\pp{(A_{1}+A_{2})x(t)+(B_{1}+B_{2})v(t)}dt$$
is exactly controllable to $\mathbb{R}^{d}$-targets (with $v=\mathbb{E}\pp{u}$), which is equivalent to the invertibility of
$$\mathcal{G}_{2}=\displaystyle\int_0^Te^{-t(A_1+A_2)}(B_{1}+B_{2})(B_{1}+B_{2})^*e^{-t(A_1+A_2)^*}dt.$$
\indent We will enforce the following stronger assumption, partly inspired by what happens in the non-mean-field framework.
\begin{ass}\label{Gramian}
We ask
$$\mathcal{G}=\left(\begin{array}{cc}
B_1 & 0 \\
0 & \displaystyle\int_0^T\mathcal{Y}_2(t)(B_{1}+B_{2})(B_{1}+B_{2})^*\mathcal{Y}^*_2(t)dt
\end{array}\right)$$to be invertible, with $\mathcal{Y}_2(t)=e^{-t(A_1+A_2)}$, for $t\geq 0$.
\end{ass}

\indent We can now give the main result of the section providing an affirmative answer to \textbf{Q2}.
\begin{theorem}\label{ThMainQ2}
Under Assumption \ref{Gramian}, for $p>1$ the system \eqref{x} is exactly controllable to $\mathbb{L}^p_{\mathcal{F}_T}\pr{\Omega;\mathbb{R}^{d}}$-targets by using $\mathcal{U}^{p}_{r}\subset\mathcal{U}^{p}$ controls.
\end{theorem}
\begin{proof}
The idea of the proof is quite simple. In order to get to the final condition, one employs a BSDE and leave aside the control. Second, using the control, one compensates the Brownian part appearing in the BSDE, together with the initial data coming from the same BSDE.

\indent We consider $x_0\in\mathbb{R}^d$, and $\xi\in\mathbb{L}^p_{\mathcal{F}_T}\pr{\Omega;\mathbb{R}^{d}}$. First, we consider the BSDE
\begin{equation}\label{Y_{1}}
\left\{\begin{aligned}
&dY_{1}(t)= \left(A_1 Y_{1}(t)+A_2 \mathbb{E}\pp{Y_{1}(t)}\right)d t +(C \mathbb{E}\pp{Y_{1}(t)}+Z_{1}(t))d W(t),\ t\leq T,\\
&Y_{1}(T)=\xi- \mathbb{E}\pp{\xi}.
\end{aligned}\right.
\end{equation}
It is clear that $\mathbb{E}\pp{Y_{1}}$ satisfies a linear deterministic equation with final data 0, thus $\mathbb{E}\pp{Y_{1}(t)}=0$, for all $t\in[0,T]$. The solution $(Y_{1},Z_{1})\in\mathbb{L}^p_{\mathbb{F}}\pr{\Omega;C\pr{\pp{0,T};\mathbb{R}^{d}}}\times\mathbb{L}^p_{\mathbb{F}}\pr{\Omega;\mathbb{L}^2\pr{\pp{0,T};\mathbb{R}^{d}}}$. Second, we set
\begin{eqs}\label{u}
u_{0}(t):=&-(B_{1}+B_{2})^*\mathcal{Y}^*_2(t)\mathcal{G}_{2}^{-1}(-e^{-T(A_1+A_2)}\mathbb{E}\pp{\xi}+x_0-Y_{1}(0)).
\end{eqs}
Next, we note that $u_{0}$ is deterministic and we consider the BSDE
\begin{equation}\label{Y_{2}}
\left\{\begin{aligned}
&dY_{2}(t)= \left(A_1 Y_{2}(t)+A_2 \mathbb{E}\pp{Y_{2}(t)}+B_1 u_{0}(t)+B_2 \mathbb{E}\pp{u_{0}(t)}\right)d t \\
&\quad \hspace{1.5cm} +(C \mathbb{E}\pp{Y_{2}(t)}+Z_{2}(t))d W(t),\ t\in[0,T],\\
&Y_2(T)=\mathbb{E}\pp{\xi}.
\end{aligned}\right.
\end{equation}
Notice that $\mathbb{E}\pp{u_{0}(t)}=u_{0}(t)$ (since $u_{0}(t)$ is deterministic). However, we write the equation like this to make the drift similar to the one in \eqref{x} for the forward equation.
One easily checks that $Y_{2}(0)=\mathbb{E}\pp{Y_{2}(0)}$ is given by
\begin{equation}\label{Y_{2}(0)}\begin{split}
Y_{2}(0)&=\mathcal{Y}_2(T)\mathbb{E}\pp{\xi}-\displaystyle\int_0^T\mathcal{Y}_2(t)(B_{1}+B_{2})u_{0}(t)dt=x_0-Y_{1}(0).
\end{split}
\end{equation}
To end the proof, we need to compensate the $Z$ terms. For this, we proceed as follows. We consider the fundamental matrix equation
\begin{equation}
\left\{\begin{aligned}
&d\Phi(t)=A_{1}\Phi(t)dt,\ t\in[0,T],\\
&\Phi(0)=I_{d},
\end{aligned}\right.
\end{equation}
whose inverse satisfies
\begin{equation}
\left\{\begin{aligned}
&d\Psi(t)=-\Psi(t)A_{1}dt,\ t\in[0,T],\\
&\Psi(0)=I_{d}.
\end{aligned}\right.
\end{equation}
Both solutions belong to $C\pr{[0,T];\mathbb{R}^{d\times d}}$. The solution to
\begin{equation}\label{x(t)}\begin{cases}
&dx(t)=\pr{A_{1}x(t)+B_{1}u(t)}dt-\pr{Z_{1}(t)+Z_{2}(t)}dW(t),\ t\geq 0,\\&x(0)=0,\end{cases}
\end{equation}
is given by
$$x(t)=\Phi(t)\int_0^t\Psi(s)B_{1}u(s)ds-\Phi(t)\int_0^t\Psi(s)\pr{Z_{1}(s)+Z_{2}(s)}dW(s).$$
Using Lemma \ref{representation} with
$$\displaystyle\int_0^T\Psi(s)\pr{Z_{1}(s)+Z_{2}(s)}dW(s)\in\mathbb{L}^p_{\mathcal{F}_T}\pr{\Omega;\mathbb{R}^{d}},$$
we see that there exists $v(\cdot)\in\mathbb{L}^1_{\mathbb{F}}\pr{\pp{0,T};\mathbb{L}^p\pr{\Omega;\mathbb{R}^{d}}}$,
such that
$$
\int_0^Tv(t)dt=\int_0^T\Psi(s)\pr{Z_{1}(s)+Z_{2}(s)}dW(s)\in\mathbb{L}^p_{\mathcal{F}_T}\pr{\Omega;\mathbb{R}^{d}}, \ \mathbb{P}\text{-a.s.}
$$
Note that $v(\cdot)-\mathbb{E}\pp{v(\cdot)}$ satisfies the same equality and belongs to $\mathbb{L}^1_{\mathbb{F}}\pr{\pp{0,T};\mathbb{L}^p\pr{\Omega;\mathbb{R}^{d}}}$. We set
$$u_{2}(t):=B_{1}^*(B_{1}B_{1}^*)^{-1}\Phi(t)v(t),$$
to get a process belonging to $\mathbb{L}^1_{\mathbb{F}}\pr{\pp{0,T};\mathbb{L}^p\pr{\Omega;\mathbb{R}^{n}}}$ and such that the solution of \eqref{x(t)} associated to $u=u_{2}-\mathbb{E}\pp{u_{2}}$ satisfies
$$
\left\{\begin{aligned}
&dx(t)= \left(A_1 x(t)+B_1(u_{2}(t)-\mathbb{E}\pp{u_{2}(t)})\right)dt\\&\quad\quad\quad\quad\quad-\pr{Z_{1}(t)+Z_{2}(t)}dW(t),\ t\in[0,T],\\
&x(0)=0,\ x(T)=0,\ \mathbb{E}\pp{x(t)}=0,\ \forall t\in[0,T],
\end{aligned}\right.
$$
or, again,
\begin{equation}\label{SDEx}
\left\{\begin{aligned}
&dx(t)=\left(A_1 x(t)+A_2\mathbb{E}\pp{x(t)}+ B_1(u_{2}(t)-\mathbb{E}\pp{u_{2}(t)})\right)dt\\
&\quad \hspace{1cm} +\left(C\mathbb{E}\pp{x(t)}-Z_{1}(t)-Z_{2}(t)\right)dW(t),\ t\in[0,T],\\
&x(0)=0,\ x(T)=0.
\end{aligned}\right.
\end{equation}
It follows from \eqref{Y_{1}}, \eqref{Y_{2}} and \eqref{SDEx} that $\mathcal{X}:=Y_{1}+Y_{2}+x$ obeys the equation
\begin{equation}\notag
\left\{\begin{aligned}
&d\mathcal{X}(t)= \left(A_1 \mathcal{X}(t)+A_2\mathbb{E}\pp{\mathcal{X}(t)}+B_1u(t)+B_2\mathbb{E}\pp{u(t)}\right)dt\\
&\quad \hspace{0.8cm} +C\mathbb{E}\pp{\mathcal{X}(t)}dW(t),\ t\in[0,T],\\
&\mathcal{X}(0)=x_0,\ \mathcal{X}(T)=\xi,
\end{aligned}\right.
\end{equation}
with $u=u_{2}-\mathbb{E}\pp{u_{2}}+u_{0}\in\mathcal{U}^{p}_{r}$. Our proof is complete.
\end{proof}
\begin{remark}
This result can be generalized, under further assumptions, to deal with $D_1,D_2\neq 0$ and time-inhomogeneous coefficients, in the spirit of \cite[Theorem 20]{goreac2021improved}.
\end{remark}
\section{\textbf{Q3.} Exact terminal controllability to normal laws}\label{Sec5}
Let us now focus on the case when $D_1=0$, thus voiding any attempt for exact controllability with $\mathbb{L}^2$-controls, and $B_1=0$, thus failing to comply with Assumption \ref{Gramian} and voiding our attempt for exact controllability with $\mathbb{L}^p$-controls. We further assume that the Brownian filtration is enriched with a large enough $\mathcal{F}_0$ capable of supporting all the laws in the 2-Wasserstein space $\mathcal{P}_2(\mathbb{R}^d)$.
\subsection{What can one do with first-order moments?}\label{Sec5a}
Let us consider the scalar case, $d=1$ and consider the scalar equation.
\begin{equation}\label{SDEd=1}
dx(t)=\pp{a_1x(t)+a_2\mathbb{E}\pp{x(t)}+b\mathbb{E}\pp{u(t)}} dt+\delta\mathbb{E}\pp{u(t)} dW(t),\ x(0)=x_0 \in \mathbb{R}.
\end{equation}
When $\delta\neq 0$, the extra term $C\mathbb{E}\pp{x(t)}$ in the diffusion coefficient in the equation \eqref{SDE1} can be dealt with by setting $v:=c_2\delta^{-1}x+u$ and changing the drift coefficient $a_2$ accordingly.\\
\indent A simple glance at \eqref{SDEd=1} shows that $\mathbb{P}_{x(t)}$ is a Gaussian law; alternatively, the reader may take a look at \cite[Section 3.2]{BGL_2023}. On the other hand, Gaussian laws are completely determined by their first and second order moments. Now $x_1(t):=\mathbb{E}\pp{x(t)}$ follows a simple equation $$dx_1(t)=\pp{(a_1+a_2)x_1(t)+bv(t)}dt,\ x_1(0)=x_0 \in \mathbb{R},$$with $v(t):=\mathbb{E}\pp{u(t)}$, an equation
which is exactly/exactly null/approximately controllable if and only if $b\neq 0$. This equation is, however, exactly terminal controllable even if $b=0$ (just take any $v$, then pick the initial condition $x$ accordingly). Now we assume $b\neq 0$.\\
Furthermore, the second order moment $x_2(t):=\mathbb{E}\pp{\abs{x(t)}^2}$ satisfies
$$dx_2(t)=\pr{2a_1x_2(t)+2a_2x_1^2(t)+2bx_1(t)v(t)+\delta^2v^2(t)} dt,$$while
$$dx_1^2(t)=\pp{2(a_1+a_2)x_1^2(t)+2bx_1(t)v(t)} dt,$$so that the variance satisfies
$$dVar(x(t))=\pr{2a_1 Var(x(t))+\delta^2v^2(t)} dt.$$
It follows that, starting from a deterministic initial value $x(0)$,
$$Var(x(T))=\delta^2\displaystyle\int_0^Te^{2a_1(T-t)}v^2(t) dt.$$
Let us consider $x_0\in\mathbb{R}$ and see that, for any $\zeta\in\mathbb{R}$,\begin{equation}\label{v}v(t):=e^{-(a_1+a_2)(T-t)}\pp{\frac{\alpha-x_0e^{(a_1+a_2)T}}{Tb}+\zeta\pr{\mathbf{1}_{[\frac{T}{2},T]}(t)-\mathbf{1}_{[0,\frac{T}{2})}(t)}}\end{equation} controls $x_1(t)\left(=\mathbb{E}\pp{x(t)}\right)$ from $x_0$ to $\alpha$. One has
\begin{align*}Var(x(T))=\delta^2\psi(T)\Big[&\pr{\frac{\alpha-x_0e^{(a_1+a_2)T}}{Tb}-\zeta}^2+e^{a_2T}\pr{\frac{\alpha-x_0e^{(a_1+a_2)T}}{Tb}+\zeta}^2\Big],\\
\textnormal{where } \psi(T)&:=\begin{cases}\frac{e^{-a_2T}-e^{-2a_2T}}{2a_2},&\textnormal{ if } a_2\neq 0,\\
\frac{T}{2},&\textnormal{ otherwise}.\end{cases}\end{align*}
It follows that, starting from $x_0\in\mathbb{R}$, and using the aforementioned control family $v$ defined in \eqref{v}, at time $T$, one can obtain all normal laws $\mathcal{N}\pr{\alpha,\beta^2}$, for all $$\beta^2\geq \delta^2\psi(T)\frac{4}{1+e^{-a_2T}}\pr{\frac{\alpha-x_0e^{(a_1+a_2)T}}{Tb}}^2.$$In particular, one covers all $\beta^2\geq 0$ when starting at $x_0:=\alpha e^{-(a_1+a_2)T}$. This is also true when $b=0$, since the associated term containing $\frac{\alpha-x_0e^{(a_1+a_2)T}}{Tb}$ can be set to $0$ as well.
We have proven the following definition-characterization for the case $d=1$.
\begin{theorem}[Theorem-Definition]
When $D_1=B_1=0_{d\times n}$, the reachable set at time $T$ of \eqref{SDE1} in law is included in the family of normal distributions, i.e.,
\begin{equation}\begin{split}
\label{ReachLaw}
&\mathcal{R}each^\mathbb{P}_T(x):=\set{\mathbb{P}_{X^{x,u}(T)}\ :\ u\in\mathbb{L}^2_{\mathbb{F}}\pr{\Omega\times\pp{0,T};\mathbb{R}^n}}\\&\quad\quad \subset\mathcal{G}auss:=\set{\mathcal{N}(\alpha,\Sigma)\ : \alpha\in\mathbb{R}^d,\ \Sigma\in\mathcal{S}_+^d},
\end{split}
\end{equation}where $X^{x,u}$ is the solution of SDE \eqref{SDE1} with initial value $x\in\mathbb{R}$ (see also \cite[Section 3.2]{BGL_2023}). Here, $\mathcal{S}_+^d$ denotes the set of positive semidefinite matrices in $\mathbb{R}^{d\times d}$.\\
The equation \eqref{SDE1} is said to be \emph{exactly terminal controllable to normal laws} (ETCNL) if $\underset{x\in\mathbb{R}^d}{\cup}\mathcal{R}each^\mathbb{P}_T(x)=\mathcal{G}auss.$
For $d=1$, the system \eqref{SDEd=1} is ETCNL if and only if $\delta\neq 0$.
\end{theorem}
\subsection{ETCNL in the Multi-dimensional Case}\label{Sec5b}
We are now dealing with the equation
\begin{equation}\label{SDE1'0}
\left\{\begin{aligned}
dX(t)=&\pr{A_1X(t)+A_2\mathbb{E}\pp{X(t)}+B_2\mathbb{E}\pp{u(t)}}dt\\
&\ \ +\pr{C\mathbb{E}\pp{X(t)}+D_2\mathbb{E}\pp{u(t)}}dW(t),\ t\in[0,T],\\
X(0)=&X_0\in \mathbb{R}^d.
\end{aligned}\right.
\end{equation}
We have the following theorem characterizing the ETCNL and providing an affirmative answer to \textbf{Q3}.
\begin{theorem}\label{ThMainQ3}
1) If $Range(C)\subset Range(D_2)$ and the system \eqref{SDE1'0} is exactly terminal controllable to normal laws, then \begin{eqs}\label{rankNETCL}
\ Rank\pr{\pp{D_2\ \vdots\ A_1D_2\ \vdots\ \ldots \ \vdots\ A_1^{d-1}D_2}}&=d.\end{eqs}
\indent 2) If one enforces the slightly stronger condition $Rank(D_2)=d$, then the system \eqref{SDE1'0} is exactly terminal controllable to normal laws.
\end{theorem}
\begin{proof}
1) For the necessary condition, we can assume, without loss of generality, that $C=0$ (Indeed, under the range condition, one is able to find a matrix $N\subset \mathbb{R}^{n\times d}$ such that $D_2N=C$ and one changes the deterministic control $u:=\mathbb{E}[u]$ to $v=u+N\mathbb{E}[X]$. The coefficients $A_1, B_2$ and $D_2$ remains unchanged, while $A_2$ becomes $\tilde{A}_2:=A_2-B_2N$).
Reasoning on the expectation of $X$, as we have already done in the one-dimensional case, we have the deterministic equation for $X_1(t)=\mathbb{E}[X(t)]$: $$dX_1(t)=\pp{(A_1+A_2)X_1(t)+B_2v(t)}dt,\ t\geq 0,\ X_1(0)=X_0\in \mathbb{R}^d,$$
where $v(t):=\mathbb{E}[u(t)]$.
By writing the tensor equation, it follows that
\begin{align*}d\pp{X_1X_1^*}(t)=\Big[&(A_1+A_2)\pp{X_1X_1^*}(t)+\pp{X_1X_1^*}(t)(A_1+A_2)^*\\
&+X_1(t)v^*(t)B_2^*+B_2v(t)X_1^*(t)\Big] dt.\end{align*}
Furthermore, if $X_2(t):=\mathbb{E}\pp{X(t)X^*(t)}$, then
\begin{align*}dX_2(t)=&\Big\{A_1X_2(t)+X_2(t)A_1^*+A_2\pp{X_1X_1^*}(t)+\pp{X_1X_1^*}(t)A_2^*\\&\quad \quad +X_1(t)v^*(t)B_2^*+B_2v(t)X_1^*(t)+D_2(vv^*)(t)D_2^*\Big\}dt.\end{align*}
For $Cov(t):=Cov(X(t))=(X_2-X_1X_1^*)(t)$, one gets
$$dCov(t)=\pr{A_1Cov(t)+Cov(t)A_1^*+D_2(vv^*)(t)D_2^*} dt.$$
The variation of constants formula gives
$$Cov(T)=\int_0^Te^{(T-t)A_1}D_2(vv^*)(t)D_2^*e^{(T-t)A_1^*} dt.$$If the rank condition \eqref{rankNETCL} on $(A_1,D_2)$ is not fulfilled, then it is easy to find a unitary vector $a\in\mathbb{R}^d$ that is orthogonal on all $A_1^kD_2$, hence orthogonal on $e^{sA_1}D_2$. Then $a^*Cov(T)a=0$, for all controls $v$ which means that $aa^*$ cannot be obtained as a covariance. We have our contradiction.\\
\indent 2) Similarly, since one requires $Rank(D_2)=d$, one can, once again, assume for \eqref{SDE1'0} that $C=0$. Let us fix $\Sigma\in\mathcal{S}_+^d$. If $\Sigma$ can be reached with a deterministic control $v$ such that $Cov(T)=\Sigma$, then one simply sets, for this control,
$$X_1(0)=x:=e^{-T(A_1+A_2)}\alpha-\displaystyle\int_0^Te^{-t(A_1+A_2)}B_2v(s)ds$$
to have the average $\alpha$ for the expectation of $X(T)$, that is, $\mathbb{E}[X(T)]=X_1(T)=\alpha$. It follows that one only needs to show how $\Sigma$ can be obtained.\\
By the spectral decomposition theorem,
$$\Sigma=\sum_{1\leq k\leq d}\mu_k\zeta_k\zeta^*_k,$$
where $\mu_k$ are the non-negative eigenvalues and $\zeta_k\in\mathbb{R}^d$ the associated eigenvectors. Since $D_2$ is assumed to be of full rank, for every $t\geq 0$, there exists a solution $v_k(t)$ of
$$D_2v_k(t)=\sqrt{\frac{d\mu_k}{T}}e^{(t-T)A_1}\zeta_k,$$with a measurable selection (one can take $v_k(\cdot):=D_2^*\pr{D_2D_2^*}^{-1}\sqrt{\frac{d\mu_k}{T}}e^{(\cdot-T)A_1}\zeta_k$). We define the measurable function
$$v(t):=\sum_{1\leq k\leq d}v_k(t)\mathbf{1}_{[\frac{(k-1)T}{d},\frac{kT}{d})}(t),\ t\geq 0.$$ Then the $Cov(T)$ associated to this control satisfies
\begin{align*}
Cov(T)=&\int_0^Te^{(T-t)A_1}D_2(vv^*)(t)D_2^*e^{(T-t)A_1^*} dt\\
&\hspace{-0.3cm}=\sum_{1\leq k\leq d}\int_{\frac{(k-1)T}{d}}^{\frac{kT}{d}}\frac{d\mu_k}{T}\zeta_k\zeta^*_k dt=\Sigma.
\end{align*}
Our proof is now complete, since we have shown how to obtain $X_T$ such that\\
$\mathbb{P}_{X_T}=\mathcal{N}(\alpha,\Sigma).$
\end{proof}

\begin{remark}
1) If $d=n,\ D_2$ commutes with $A_1$, and $a\neq 0$ is such that $$a^*e^{(T-t)A_1}D_2=0,$$
for some $t$, then
$$a^*e^{(T-s)A_1}D_2=a^*e^{(T-t)A_1}D_2e^{(t-s)A_1}=0,$$
for all $s\geq 0$. Then \eqref{rankNETCL} denies the existence of such $a$. As a consequence, necessarily, $D_2$ must be of full rank.\\
\indent 2) When looking back at our initial assumption of $Range(C)\subset Range(D_2)$, we understand that the condition of full range on $D_2$ makes that this assumption is not a real restriction.\\
\indent 3) The construction in the one-dimensional case is more flexible as it also describes the panel of variances reachable starting from $x$ and using the $\pm 1$ symmetric controls. Such constructions are obviously available for the $d$-dimensional case and give a hint to a question of \emph{normal exact controllability in law} (instead of the weaker \emph{terminal} one).
\end{remark}
\subsection{How Does ETCNL Compare to Other Types of Controllability?}\label{Sec5c}

The short answer: It is a \textbf{completely new concept}.\\

\indent \textbf{(I)} While it is implied by exact (terminal) controllability in $\mathbb{L}^2_{\mathcal{F}_T}\pr{\Omega;\mathbb{R}^d}$, it is a class disjoint from the approximate controllability concept in $\mathbb{L}^2_{\mathcal{F}_T}\pr{\Omega;\mathbb{R}^d}$.
Indeed, since $\mathcal{G}auss$ is closed with respect to the weak-* convergence of probability measures, if one thinks about Wasserstein metric, for instance, $\mathcal{G}auss$ does not approximate non Gaussian laws. As such it does not imply approximate (terminal) controllability in law (ATCL)\footnote{For the controlled solutions, we have set, in \eqref{ReachLaw}, the reachable set in law to be $\mathcal{R}each^\mathbb{P}_T(x)$. Then, exact terminal controllability in law to $\mu\in\mathcal{P}_{2}(\mathbb{R}^d)$ from $x$ is valid if $\mu\in\mathcal{R}each^\mathbb{P}_T(x).$ By analogy, the concept of approximate terminal controllability in law to $\mu$ starting at $x$
is valid if $\mu\in cl\left(\mathcal{R}each^\mathbb{P}_T(x)\right)$, where the $cl$ is the Kuratowski closure operator w.r.t. the Wasserstein distance $W_{2}$. In other words, $\mu$ is ATCL $\mu\in cl\left(\cup_{x\in\mathbb{R}^d}\mathcal{R}each^\mathbb{P}_T(x)\right)$.}. Hence, it cannot imply approximate controllability in $\mathbb{L}^2_{\mathcal{F}_T}\pr{\Omega;\mathbb{R}^d}$.\\
\indent \textbf{(II)} The 1-dimensional dynamics $dx(t)=\mathbb{E}\pp{u(t)}dt$ provides an example of a system that is approximately controllable to $0$ in $\mathbb{L}^2_{\mathcal{F}_T}\pr{\Omega;\mathbb{R}^d}$, yet it is not ETCNL.\\
\indent\textbf{(III)} A system may be ETCNL, but failing to be approximate null-controllable (although, obviously, it will be exactly terminal null-controllable).
\begin{example}\label{ExpNoApp0Ctrl}
We consider the one dimensional framework, and deal with the system \eqref{SDE1'0} with $D_2=D\neq 0$. We further assume $B_2=B\neq 0$ (which is the condition for controllability of the deterministic part, i.e., $\mathbb{E}[X(t)]$). We will show that, in this case, even though the system is ETCNL it is not approximately null-controllable.\\
We have
$$Var(T):=Var(X(T))=D^2\displaystyle\int_0^Te^{2A_1(T-t)}v^2(t)dt,$$
and
$$\mathbb{E}\pp{X(T)}=e^{(A_1+A_2)T}x+B\displaystyle\int_0^Te^{(A_1+A_2)(T-t)}v(t) dt.$$
Assuming that arbitrarily small neighborhoods of $0$ can be reached in $\mathbb{L}^2_{\mathcal{F}_T}\pr{\Omega;\mathbb{R}^d}$ starting from $x$ implies that,
$$Var(T)+\pr{\mathbb{E}\pp{X(T)}}^2=\mathbb{E}\pp{(X(T)-0)^2}$$ can be made arbitrarily small, for every initial $x\in\mathbb{R}$. However, \begin{align*}&\pr{\int_0^Te^{(A_1+A_2)(T-t)}v(t)\ dt}^2\leq \frac{1}{D^2}\int_0^Te^{2A_2(T-t)}\ dtVar(T)=\frac{e^{2A_2T}-1}{2A_2D^2}Var(T),\end{align*}where, if $A_2=0$, we change $\frac{e^{2A_2T}-1}{2A_2}$ to $T$. Then, by letting
$$y:=\displaystyle\int_0^Te^{(A_1+A_2)(T-t)}v(t) dt,$$ and
$$z:=e^{(A_1+A_2)T}x,$$
we have, if $A_2\neq 0$,
\begin{align*}
&Var(T)+\pr{\mathbb{E}\pp{X(T)}}^2\geq z^2+2Bzy+\pr{B^2+\frac{2A_2D^2}{e^{2A_2T}-1}}y^2\\&\indent \quad \hspace{2.3cm} \geq \frac{2A_2D^2}{B^2(e^{2A_2T}-1)+2A_2D^2}e^{2(A_1+A_2)T}x^2.
\end{align*}
When $A_2=0$, similar arguments yield
\begin{align*}
Var(T)+\pr{\mathbb{E}\pp{X(T)}}^2&\geq \frac{1}{B^2T+1}e^{2A_1T}x^2.
\end{align*}
It is clear that the only trajectory approximately controllable to $0$ has to start from $x=0$.
\end{example}
\subsection{A new Wasserstein set-valued BSDE}\label{Sec5d}

Example \ref{ExpNoApp0Ctrl} in dimension $d=1$, and the standard inequality employed to prove absence of approximate null controllability reveal an interesting form of BSDE, provided $Rank(D)=d$:
\begin{equation}\label{BSDE_N}
\left\{\begin{aligned}
dY(t)=&\pr{A_1Y(t)+A_2\mathbb{E}\pp{Y(t)}+BD^*\pr{DD^*}^{-1}z(t)}dt+z(t)dW(t),\ t\in\pp{0,T},\\
\mathbb{P}_{Y(T)}=&\mu.\end{aligned}\right.
\end{equation}
The final datum is a probability law belonging to the 2-Wasserstein space $\mathcal{P}_2$. To our best knowledge, this kind of equations are new. Let us give the following definition.
\begin{definition}
Given $\Gamma\subset \mathcal{G}auss$, a Gaussian set-valued solution of \eqref{BSDE_N} is a family of couples $(Y,z)$ consisting of a continuous adapted process $Y$ and a deterministic Borel measurable $z$, square integrable on $\pp{0,T}$, both with values in $\mathbb{R}^d$, such that\\
\indent 1) $\mathbb{P}_{Y(T)}\in\Gamma$;\\
\indent 2) $Y(0)$ is an $\mathcal{F}_0$-random variable with $\mathbb{P}_{Y(0)}\in \mathcal{G}auss$;\\
\indent 3) $(Y,z)$ satisfies \eqref{BSDE_N}, i.e.,
\begin{align*}
&Y(t)=Y(T)-\int_t^T\pr{A_1Y(s)+A_2\mathbb{E}\pp{Y(s)}+BD^*\pr{DD^*}^{-1}z(s)}ds\\&\indent \quad \hspace{1.5cm}-\int_t^Tz(s)dW(s),\ \mbox {for any}\ t\in\pp{0,T},\ \mathbb{P}\text{-a.s.}
\end{align*}
\end{definition}
\begin{remark}
1) If $Y(t)$ has Gaussian distribution, the same is valid for $Y(s)$ for $s\in \pp{t,T}$ whenever the equation is seen as a progressive equation with $z$ acting as a control.\\
\indent 2) We can alternatively ask $(Y,Z)$ to be a usual solution, without imposing $Z$ to be deterministic, but ask that $\mathbb{P}_{Y(\cdot)}\in\mathcal{G}auss$ (not just the starting and ending points). This is equivalent to imposing $\displaystyle\int_t^sZ(l)dW(l)$ to be Gaussian. Of course, the two definitions are not equivalent as, in this case, $Z$ may not be deterministic; constructions based on the sign $sgn(W)$ are easy to make relying on L\'evy's martingale characterization of Brownian motions.
\end{remark}
It is clear that
$$y(t):=\mathbb{E}\pp{Y(t)}\ \text{and}\ \sigma(t):=
\mathbb{E}\pp{(Y^*Y)(t)}-\pr{\mathbb{E}\pp{Y(t)}}^*\mathbb{E}\pp{Y(t)}$$
are given by
\begin{eqc}
\sigma(s)=e^{-(T-s)A_1}\sigma(T)e^{-(T-s)A_1^*}-\displaystyle\int_s^Te^{-(r-s)A_1}z(r)z(r)^*e^{-(r-s)A_1^*}dr,\\
y(s)=e^{-(T-s)(A_1+A_2)}y(T) -\displaystyle\int_s^Te^{-(r-s)(A_1+A_2)}BD^*\pr{DD^*}^{-1}z(r)dr,\ s\in [0,T].
\end{eqc}
\begin{example}\label{exp2}We consider the simplest equation
\begin{equation}
\left\{\begin{aligned}
&dY(t)=z(t)dW(t),\ t\in [0,1],\\
&\mathbb{P}_{Y(1)}=\mathcal{N}(0,1).\end{aligned}\right.
\end{equation}
\indent 1) $z$ is not unique even when $\mathbb{P}_{Y(0)}$ is fixed. To understand this, one takes $\mathbb{P}_{Y(0)}$\\
\indent \ \ \ $=\delta_0$ and notes that $z=\mathbf{1}_{\pp{0,t_0}}-\mathbf{1}_{(t_0,1]}$ does the job by simply picking $Y(1)=$\\
\indent \ \ \ $2W(t_0)-W(1)$, with $t_0$ arbitrary in $\pp{0,1}$.\\
\indent 2) On the other hand, for every $\sigma\leq 1$ and $0<s_0< T=1$, the solution with\\
\indent \ \ \ $Y(0)=0$, $z=\sqrt{\frac{\sigma}{s_0}}\mathbf{1}_{[0,s_0]}+\sqrt{\frac{1-\sigma}{T-s_0}}\mathbf{1}_{(s_0,1]}$, satisfies $\mathbb{P}_{Y(s_0)}=\mathcal{N}(0,\sigma)$ and\\
\indent \ \ \ $\mathbb{P}_{Y(1)}=\mathcal{N}(0,1)$.
It shows that $\mathbb{P}_{Y(s_0)}$ is not unique.\\
\indent 3) If $Y(0)$ is not required to be deterministic and $\mathcal{F}_0$ is independent of $W$, then\\
\indent \ \ \ any $\mathbb{P}_{Y(0)}=\mathcal{N}(0,\sigma)$ works if $\sigma\leq 1$ by picking $z=\sqrt{\frac{1-\sigma}{T}}$ (recall that $T=1$).

So, the conclusion is that we deal here with a \textbf{Wasserstein-set-valued} BSDE \eqref{BSDE_N}.
\end{example}
Let us give a complete treatment in the one-dimensional framework $d=1$. In this framework, if the coefficients are denoted by $$A_1=a_1,\ A_2=a_2,\ BD^*(DD^*)^{-1}=b,$$ then
\begin{eqc}\label{d=1treatment}
\sigma(s)=e^{-2a_1(T-s)}\sigma(T)-\displaystyle\int_s^Te^{-2a_1(r-s)}z^2(r)dr,\\
y(s)=e^{-(a_1+a_2)(T-s)}y(T)-b\displaystyle\int_s^Te^{-(a_1+a_2)(r-s)}z(r)dr,\ s\in [0,T].
\end{eqc}
\begin{theorem}\label{BSDE_1}
Assume $\mathcal{F}_0$ is rich enough. Given $\mu\in\mathcal{G}auss$ at time $T$, the set of backward reachable laws at time $s< T$ is
\begin{eqs}
\mathcal{Y}^\mu(s):=\Big\{&\mathcal{N}(y,\sigma)\ :\ \sigma \in \pp{0,e^{-2a_1(T-s)}Var(\mu)},\ y\in\pp{y_{\min}(\sigma),y_{\max}(\sigma)}\Big\},
\end{eqs}
where $Var(\mu)=\int_{\mathbb{R}}(r-\int_{\mathbb{R}}s\mu(ds))^2\mu(dr)$ is the variance, and,
\begin{eqs}\label{y}
&y_{\min}(\sigma):=e^{-(a_1+a_2)(T-s)}\displaystyle\int_{\mathbb{R}}r\mu(dr)-\abs{b}\sqrt{f(T-s)}\sqrt{e^{-2a_1(T-s)}\sigma(T)-\sigma},\\
&y_{\max}(\sigma):=e^{-(a_1+a_2)(T-s)}\displaystyle\int_{\mathbb{R}}r\mu(dr)+\abs{b}\sqrt{f(T-s)}\sqrt{e^{-2a_1(T-s)}\sigma(T)-\sigma},
\end{eqs}
where $f(s):=\begin{cases}\frac{1-e^{-2a_2s}}{2a_2},&\textnormal{ if }a_2\neq 0,\\
s,&\textnormal{ otherwise,}\end{cases}$
 $s\in [0,T]$. \end{theorem}
\begin{proof}
We sketch the proof only for $b\neq 0$ (the other case is even simpler and follows a similar argument). It is easy to show that from \eqref{d=1treatment} at time $s$, the equation $\sigma(s)=\sigma$ has solutions $z^2$ for every $\sigma\in \pp{0,e^{-2a_1(T-s)}\sigma(T)}$ and only for these values.\\
One notes that
$$f(T-s)\int_s^Te^{-2a_1(r-s)}z^2(r)dr\geq \pr{\int_s^Te^{-(a_1+a_2)(r-s)}z(r)dr}^2.$$
As a simple consequence,
$$\abs{b}\sqrt{f(T-s)\pr{e^{-2a_1(T-s)}\sigma(T)-\sigma(s)}}\geq \abs{e^{-(a_1+a_2)(T-s)}y(T)-y(s)}.$$ For every $\sigma(s)=\sigma\in \pp{0,e^{-2a_1(T-s)}\sigma(T)}$, there are two extremal solutions in the sense of \emph{stochastic ordering} of random variables $y_{\min}(\sigma)$ and $y_{\max}(\sigma)$. That is, $y_{\min}(\sigma)$ and $y_{\max}(\sigma)$ are obtained for $z(r)=\pm sgn(b)ce^{(a_1-a_2)(r-s)}$, respectively, with $c>0$ constant such that
$$e^{-2a_1(T-s)}\sigma(T)-\sigma=\int_s^Te^{-2a_1(r-s)}z^2(r)dr.$$ By setting
$$z_\alpha(r)=\pp{-\mathbf{1}_{[s,s+\alpha (T-s))}+\mathbf{1}_{[s+\alpha (T-s),T]}}(r)ce^{(a_1-a_2)(r-s)},$$ the square norm $\displaystyle\int_s^Te^{-2a_1(r-s)}z_{\alpha}^2(r)dr$ is kept but the associated $y$ covers the whole interval $\pp{y_{\min},y_{\max}}$ where $\alpha$ is from $\pp{0,1}$. The proof is now complete.
\end{proof}
\section{Conclusions}
In this paper, we have completed the approach in \cite{goreac2014controllability} by showing that the joint rank conditions on $D_1$ and $D_1+D_2$ are required for exact controllability with square-integrable controls. When this fails, we have provided the mean-field analogous of the exact controllability with less regular controls from \cite{wang2016exact}, priorly highlighted in \cite{lu2012representation}.
When the required conditions fail both on the drift and on the noise coefficient, we introduce a new notion of controllability in law tailored for mean-field systems. In addition, we introduce a class of Wasserstein set-valued BSDEs and give a complete descrition of the sets of solutions for uni-dimensional systems. The multi-dimensional case is a work-in-progress as are more general notions of controllability in law.

\end{document}